\documentclass[preprint,12pt]{elsarticle}

\usepackage{amssymb}
\usepackage{amsthm}
\newtheorem{theorem}{Theorem}
\journal{Applied Mathematics and Computation}
\begin{document}

\begin{frontmatter}



\title{On strong homogeneity of two global optimization algorithms based on statistical models of multimodal objective functions}


\author{Antanas \v{Z}ilinskas}

\address{Vilnius University\\
             Institute of Mathematics and Informatics,\\
             Vilnius, Lithuania\\
             antanas.zilinskas@mii.vu.lt}

\begin{abstract}
The implementation of global optimization algorithms, using the arithmetic of infinity, is considered. A relatively simple version of implementation is proposed for the algorithms that possess the introduced property of strong homogeneity. It is shown that the P-algorithm and the one-step Bayesian algorithm are strongly homogeneous.
\end{abstract}

\begin{keyword}
Arithmetic of infinity\sep Global optimization \sep Statistical models
\end{keyword}

\end{frontmatter}

% \linenumbers

%\section{}
%\label{}
\section{Introduction}
\label{sec:intro}

Global optimization problems are considered where the computation of objective function values,
using the standard computer arithmetic, is problematic because of either underflows or overflows. A perspective means for solving
such problems is the arithmetic of infinity \cite{YS08,YS09,YS10}. Besides fundamentally new problems of minimization of functions whose computation involves infinite or infinitesimal values, the arithmetic of infinity can be also very helpful for the cases where the computation of objective function values is challenging because of the involvement of numbers differing in many orders of magnitude. For example, in some problems of statistical inference \cite{AZ11,ZZ10}, the values of operands, involved in the computation of objective functions, differ by more than a factor of $10^{200}$.

The arithmetic of infinity can be applied to the optimization of challenging objective functions in two ways. First, the optimization algorithm can be implemented in the arithmetic of infinity.  Second, the arithmetic of infinity can be applied to scale the
objective function values to be suitable for processing by a conventionally implemented optimization algorithm. The second case is simpler to apply, since
the arithmetic of infinity should be applied only to the scaling of function values. If both implementation versions of the algorithm perform identically with respect to the generation of sequences of points where the objective function values are computed, the algorithm is called strongly homogeneous. In the present paper, we show that both implementation versions - of the P-algorithm and of the one-step Bayesian algorithm - are strongly homogeneous.

To be more precise, let us consider two objective functions $f(x)$ and $h(x)$, $x\in A \subseteq R^d$ differing only in scales of function values, i.e. $h(x)=af(x)+b$ where $a$ and $b$  are constants that can assume not only finite but also infinite and infinitesimal values expressed by numerals introduced in \cite{YS08,YS09}. In its turn,  $f(x)$ is defined by using the traditional finite arithmetic. The sequences of points generated by an algorithm, when applied to these functions, are denoted by $x_i,\,i=1,2,\ldots,$ and $v_i,\,i=1,2,\ldots,$ respectively. The algorithm that generates the identical sequences $x_i=v_i,\,i=1,2,\ldots,$ is called strongly homogeneous. A weaker property of algorithms is considered  in \cite{ES09,SS00}, where the algorithms that generate the identical sequences for the functions $f(x)$ and $h(x)=f(x)+b$ are called homogeneous. Since the proper scaling of function values by translation alone is not always possible, in the present paper we consider invariance of the optimization results with respect to a more general (affine) transformation of the objective function values.

\section{Description of the P-algorithm}
\label{sec:P}

Let us consider the minimization problem
\begin{equation}\label{eq;1}
   \min_{x\in A} f(x),\, A\subseteq R^d,
\end{equation}
where the multimodality of the objective function $f(x)$ is expected. Although the properties of the feasible
region are not essential in a further analysis, for the sake of explicitness, $A$ is assumed to be a hyper-rectangle.
For the arguments justifying the construction of global optimization algorithms using statistical models of objective functions, we refer to \cite{SS00,AZ82,ZZ08}.
 Global optimization algorithms based on statistical models implement the ideas of the theory of rational decision making under uncertainty \cite{TZ89}. The P-algorithm
is constructed in \cite{AZ85} stating the rationality axioms in the situation of selection of a point of current computation of the value of $f(x)$; it follows from the axioms that  a point should be selected where the  probability to improve the current best value is maximal.

To implement the P-algorithm,  Gaussian stochastic functions are used mainly
 because of their computational advantages; however such type of statistical models is justified axiomatically and by the results of a psychometric experiment \cite{TZ89,ZZ08,AZ85}. Application for a statistical model of a non-Gaussian stochastic function would imply at least serious implementation difficulties. Let  $\xi(x)$ be the Gaussian stochastic function with mean value $\mu$, variance $\sigma^2$, and correlation function $\rho(\cdot,\cdot)$. The choice of the correlation function normally is based on the supposed properties of the aimed objective functions, and the properties of the corresponding stochastic function, e.g. frequently used correlation functions are $\rho(x_i,x_j)=\exp(-c||x_i-x_j||)$, $\rho(x_i,x_j)=\exp(-c||x_i-x_j||^2)$. The parameters $\mu$ and $\sigma^2$ should be estimated using a sample of the objective function values.

Let $y_i=f(x_i)$ be the function values computed during the previous $n$ minimization steps.
By the P-algorithm \cite{TZ89,AZ85} the next function value is computed at the point of maximum probability to overpass the aspiration level $y_{on}$:
\begin{equation}\label{P}
    x_{n+1}=\arg \, \max_{x\in A}{\textbf P} \{\xi(x)\leq
    y_{on}|\xi(x_i)=y_i,\,i=1,...,n\}.
\end{equation}
 Since $\xi(x)$ is the
Gaussian stochastic function, the maximization in (\ref{P}) can be
reduced to the maximization of
\begin{eqnarray}\label{P1}
\frac{y_{on}-m_n(x|x_i,y_i)}{s_n(x|x_i,y_i)},
\end{eqnarray}
where $m_n(x|x_i,y_i)$ and $s_n^2(x|x_i,y_i)$ denote the conditional mean and conditional
variance of $\xi(x)$ with respect to $\xi(x_i)=y_i,\, i=1,...,n,$. The explicit formulae of $m_n(x|x_i,y_i)$ and $s_n^2(x|x_i,y_i)$ are presented below since they will be needed in a further analysis
\begin{eqnarray}\label{cond}
\nonumber  &&m_n(x|x_i,y_i) = \mu + (y_1-\mu,\ldots,y_n-\mu) \Sigma^{-1}\Upsilon^T,\\
\nonumber  &&s_n^2(x|x_i,y_i) = \sigma^2(1- \Upsilon\Sigma^{-1}\Upsilon^T),\\
  &&\Upsilon=(\rho(x_1,x),\ldots,\rho(x_n,x)),\;\Sigma=\left(
                                                         \begin{array}{ccc}
                                                           \rho(x_1,x_1) & \ldots & \rho(x_1,x_n) \\
                                                           \ldots & \ldots & \ldots \\
                                                           \rho(x_n,x_1) & \ldots & \rho(x_n,x_n) \\
                                                         \end{array}
                                                       \right).
\end{eqnarray}

\section{Evaluation of the influence of scaling on the search by the P-algorithm}
\label{sec:SHP}

To evaluate the influence of data scaling on the whole optimization process, two objective functions are considered: $f(x)$ and $\phi (x)=a\cdot f(x)+b$, where $a$ and $b$ are constants. Let us assume that the first $n$ function values were computed for both functions at the same points $(x_i,\,i=1,...,n)$. The next points of computation of the values of $f(\cdot)$ and $\phi(\cdot)$ are denoted by $x_{n+1}$ and $v_{n+1}$. We are interested in the strong homogeneity of the P-algorithm, i.e. in the equality $x_{n+1}=v_{n+1}$.

The parameters of the stochastic function, estimated using the same method but different function values, normally are different. The estimates of $\mu$ and $\sigma^2$, obtained using the data $(x_i,y_i=f(x_i),\,i=1,...,n)$ and $(x_i,z_i=\phi(x_i),\,i=1,...,n)$,   are denoted as $\bar{\mu}, \bar{\sigma}^2$ and $\tilde{\mu}, \tilde{\sigma}^2$, respectively. It is assumed that $\tilde{\mu}=a\bar{\mu}+b$ and $\tilde{\sigma}^2=a^2\bar{\sigma}^2$; as shown below, this natural assumption is satisfied for the two most frequently used estimators.

 Obviously, the unbiased estimates of $\mu$ and of $\sigma^2$,  $\tilde{\mu}=\frac{1}{k}\sum_1^kz_i$, and $\tilde{\sigma}^2=\frac{1}{k-1}\sum_1^k(\tilde{\mu}-z_i)^2$, satisfy the assumptions made. Although those estimates are well justified only for independent observations, they sometimes (especially when only a small number ($k$) of observations is available)  are used also for rough estimation of the parameters $\mu$ and of $\sigma^2$ despite the correlation between the $\{z_i\}$.

 The maximum likelihood estimates also satisfy the assumptions:
\begin{eqnarray}\label{likl}
  &&  (\tilde{\mu},\tilde{\sigma}^2)=\arg \max\limits_{\mu,\sigma^2}^{}\frac{1}{(2\pi)^{n/2} |\Sigma|^{1/2}\sigma^n}\exp\left(-\frac{(y-\mu I)\Sigma^{-1}(y-\mu I)^T}{2\sigma^2}\right),
\end{eqnarray}
where $y=(y_1,\ldots,y_n)^T$, and $I$ is the $n$ dimensional unit vector.

It is easy to show that the maximum likelihood estimates implied by (\ref{likl}) are equal to
\begin{eqnarray}
 \label{likl1} \tilde{\mu} &=& \frac{\sum_{i=1}^n\sum_{i=1}^n z_i \rho(x_i,x_j)}{\sum\sum \rho(x_i,x_j)} \\
 \label{likl2} \tilde{\sigma}^2 &=& \frac{1}{n} \left((y-\tilde{\mu} I)\Sigma^{-1}(y-\tilde{\mu} I)^T\right).
\end{eqnarray}
It follows from (\ref{likl1}) and (\ref{likl2}) that $\tilde{\mu}=a\bar{\mu}+b$, and $\tilde{\sigma}^2=a^2\bar{\sigma}^2$ correspondingly.

The aspiration levels are defined depending on the scales of function values: $ y_{on}=\min\limits_{i=1,...,n} y_i-\varepsilon\bar{\sigma}$, $ z_{on}=\min\limits_{i=1,...,n} z_i-\varepsilon\tilde{\sigma}$.

\begin{theorem}
The P-algorithm, based on the Gaussian model with estimated parameters, is strongly homogeneous.
\end{theorem}
\begin{proof}
According to the definition of $v_{k+1}$ the following equalities are valid
\begin{eqnarray}\label{v}
\nonumber &&v_{n+1}=\arg \, \max_{x\in A} \frac{z_{on}-m_n(x|x_i,z_i)}{s_n(x|x_i,z_i)}\\
   &&=\arg \, \max_{x\in A} \frac{\min\limits_{i=1,...,n}^{}z_i-\varepsilon\tilde{\sigma}-(\tilde{\mu} + (z_1-\tilde{\mu},\ldots,z_n-\tilde{\mu}) \Sigma^{-1}\Upsilon^T)}
   {\tilde{\sigma}\sqrt{(1- \Upsilon\Sigma^{-1}\Upsilon^T)}}.
\end{eqnarray}

Taking into account the relation between $z_i$ and $y_i$ and the corresponding relations between the  estimates of $\mu$ and $\sigma$,  equalities (\ref{v}) can be extended as follows
\begin{eqnarray}\label{v1}
\nonumber &&v_{n+1}=\arg \, \max_{x\in A} \frac{\min\limits_{i=1,...,n}^{}ay_i-a\varepsilon\bar{\sigma}- a(y_1-\bar{\mu},\ldots,y_n-\bar{\mu}) \Sigma^{-1}\Upsilon^T}
   {a\bar{\sigma}\sqrt{(1- \Upsilon\Sigma^{-1}\Upsilon^T)}}\\
\nonumber   &&=\arg \, \max_{x\in A} \frac{\min\limits_{i=1,...,n}^{}y_i-\varepsilon\bar{\sigma}- (y_1-\bar{\mu},\ldots,y_n-\bar{\mu}) \Sigma^{-1}\Upsilon^T}
   {\bar{\sigma}\sqrt{(1- \Upsilon\Sigma^{-1}\Upsilon^T)}}\\
   &&=\arg \, \max_{x\in A} \frac{y_{on}-m_n(x|x_i,y_i)}{s_n(x|x_i,y_i)}=x_{n+1}.
\end{eqnarray}
The equality between $v_{n+1}$ and  $x_{n+1}$ means that the sequence of points generated by the P-algorithm is invariant with respect to the scaling of the objective function values. The strong homogeneity of the P-algorithm is proven.
\end{proof}

 As shown in \cite{Gut01,AZ10}, the P-algorithm and the radial basis function algorithm are equivalent under very general assumptions. Therefore the  statement on the strong homogeneity of the P-algorithm is also valid for the radial basis function algorithm.

\section{Evaluation of the influence of scaling on the search by the one-step Bayesian algorithm}
\label{sec:SHB}

Statistical models of objective functions are also used to construct Bayesian algorithms \cite{Mock72, Mock88}. Let a Gaussian stochastic function $\xi(x)$ be chosen for the statistical model as in Section~\ref{sec:P}. An implementable version of the Bayesian algorithm is the so called one-step Bayesian algorithm defined as follows:
\begin{eqnarray}\label{B}
  x_{n+1} &=& \arg \max_{x \in A} \mathbf{E}\{\max(y_{on}-\xi(x),0)|\xi(x_i)=y_i, \,i=1,\ldots,n\}.
\end{eqnarray}

\begin{theorem}
The one-step Bayesian algorithm, based on the Gaussian model with estimated parameters,  is strongly homogeneous.
\end{theorem}
\begin{proof}
 The value of the objective function is computed by the one-step Bayesian  algorithm at the point of maximum  average improvement (\ref{B}). The formula of conditional mean in (\ref{B}) can be rewritten as follows
\begin{eqnarray}\label{B1}
\nonumber  && \mathbf{E}\{\max(y_{on}-\xi(x),0)|\xi(x_i)=y_i, \,i=1,\ldots,n\}= \\
  && =\int_{-\infty}^{y_{on}}(y_{on}-t)p(t|m_n(x|x_i,y_i),s_n^2(x|x_i,y_i))dt,
\end{eqnarray}
where $p(t|\mu,\sigma^2)$ denotes the Gaussian probability density with the mean value $\mu$ and variance $\sigma^2$. For simplicity, we use in this formula and hereinafter the traditional symbol $\infty$. Obviously, when one starts to work in the framework of the infinite arithmetic \cite{YS08,YS09}, it should be substituted by an appropriate infinite number that has been defined a priori by the chosen statistical model.

Integration by parts in (\ref{B1}) results in the following formula
\begin{eqnarray}
\nonumber  && \mathbf{E}\{\max(y_{on}-\xi(x),0)|\xi(x_i)=y_i, \,i=1,\ldots,n\}= \\
  &&  =s_n(x|x_i,y_i) \int_{-\infty}^{\frac{y_{on}-m_n(x|x_i,y_i)}{s_n(x|x_i,y_i)}} \Pi(t) dt,
\end{eqnarray}
where $\Pi(t)$ is the Laplace integral: $\Pi(t)=\frac{1}{2\pi} \int_{-\infty}^t \exp(-\frac{\tau^2}{2})d\tau $. From the formulae (\ref{cond}), (\ref{v1}), the equalities
\begin{eqnarray}
\nonumber  && \frac{y_{on}-m_n(x|x_i,y_i)}{s_n(x|x_i,y_i)}= \frac{z_{on}-m_n(x|x_i,z_i)}{s_n(x|x_i,z_i)},\\
\nonumber  && s_n^2 (x|x_i,z_i) =a s_n^2(x|x_i,y_i),
\end{eqnarray}
follow implying the invariance of the sequence $x_1,\,x_2,\ldots ,$ generated by the one-step Bayesian algorithm with respect to the scaling of values of the objective function. The strong homogeneity of the one-step Bayesian algorithm is proven.
\end{proof}

\section{Strong homogeneity is not a universal property of global optimization algorithms}
Although the invariance of the whole optimization process with respect to affine scaling of objective function values seems very natural, not all global optimization algorithms are strongly homogeneous. For example,  the rather popular algorithm DIRECT \cite{JO93} is not strongly homogeneous. We are not going to investigate in detail the properties of DIRECT related to the scaling of objective function values. Instead an example is presented contradicting the necessary conditions of strong homogeneity.

For the sake of simplicity let us consider the one-dimension version of DIRECT. Let the feasible region (interval) be partitioned into subintervals $[a_i, b_i]$, $i=1,...,n$.  The objective function values computed at the points $c_i=(a_i+b_i)/2$ are supposed positive, $f(c_i)>0$; denote $f_{min}=\min \{f(c_1),\ldots ,f(c_n)\}$. A $j$-th subinterval is said to be potentially optimal if there exists a constant $L>0$ such that
\begin{eqnarray}
\label{p}  f(c_j)-L\Delta_j &\le& f(c_i)-L\Delta_i,\;\forall i=1,...,n ,\\
\label{a}  f(c_j)-L\Delta_j &\le& f_{min}-\varepsilon |f_{min}|,
\end{eqnarray}
where $\Delta_i=(b_i-a_i)/2$, and $\varepsilon $ is a constant defining the requested relative improvement, $0<\varepsilon <1$. All potentially optimal subintervals are subdivided at the current iteration.

Let us consider the iteration where the potentially optimal $j$-th subinterval is not the longest one.  Then  $f(c_j)\le f(c_i)$ for all $c_i$ where $\Delta_j = \Delta_i$.  Otherwise there exists a constant $L$ such that
\begin{eqnarray}
\label{L1}  L &\ge& \max_{i:\; \Delta_j>\Delta_i}\frac{f(c_j)-f(c_i)}{\Delta_j-\Delta_i},\\
\label{L2}  L &\le& \min_{i:\; \Delta_j<\Delta_i}\frac{f(c_i)-f(c_j)}{\Delta_i-\Delta_j},\\
\label{L3}  L&\ge&(f(c_j)-f_{min}+\varepsilon |f_{min}|)/\Delta_j.
\end{eqnarray}

The values $f(c_i)$ and $\Delta_i$ corresponding to the minimum in (\ref{L2}) are denoted as $f^+$ and $\Delta^+$ correspondingly, i.e.
\begin{equation}\label{L+}
    \min_{i:\; \Delta_j<\Delta_i}\frac{f(c_i)-f(c_j)}{\Delta_i-\Delta_j}=\frac{f^+-f(c_j)}{\Delta^+-\Delta_j}.
\end{equation}

Let the values of the function $\phi(x)=f(x)+\delta$ be computed at the points $c_i$, and assume that the following inequality $\delta> \delta_f/{\varepsilon}$ is valid, where
 \begin{equation}\label{d}
   \delta_f =  (f^+-f(c_j))\frac{\Delta_j}{\Delta^+-\Delta_j}-f(c_j)+(1-\varepsilon)f_{min}.
 \end{equation}

For the data related to $\phi(x)$ the following inequality holds:
\begin{eqnarray*}
  L_-&=&(\phi(c_j)-\phi_{min}+\varepsilon |\phi_{min}|)/\Delta_j \\
 &>& (f(c_j)-f_{min})/\Delta_j
   +(\varepsilon f_{min}+ \delta_f/\Delta_j)\\
&=&\frac{f^+-f(c_j)}{\Delta^+-\Delta_j}=L_+,
\end{eqnarray*}
and a constant $L$ satisfying the inequalities $L_-\le L \le L_+$ can not exist. Therefore the $j$-th subinterval for the function $\phi(x)$ is not potentially optimal because necessary conditions (analogous to (\ref{L2}) and (\ref{L3}) for the function $f(x)$) are not satisfied.

\section{Numerical Example}
To demonstrate the strong homogeneity of the P-algorithm an example of one dimensional optimization is considered. For a statistical model the stationary Gaussian stochastic function with correlation function $\rho(t)=\exp(-5t)$ is chosen. Let the values of the first objective function (say $f(x)$) computed at the points (0, 0.2, 0.5, 0.9, 1) be equal to (-0.8, -0.9, -0.65, -0.85, -0.55), and the values of the second objective function (say $\phi(x)$) be equal to (0, -0.4, 0.6, -0.2, 0.99). The graphs of the conditional mean and conditional standard deviation for both sets of data are presented in Figure~\ref{fig:1}. In the section of Figure~\ref{fig:1} showing the conditional means, the horizontal lines are drawn at the levels $y_{o4}$ and $z_{o4}$ correspondingly.

\begin{figure}
%\sidecaption
\includegraphics[width=\textwidth]{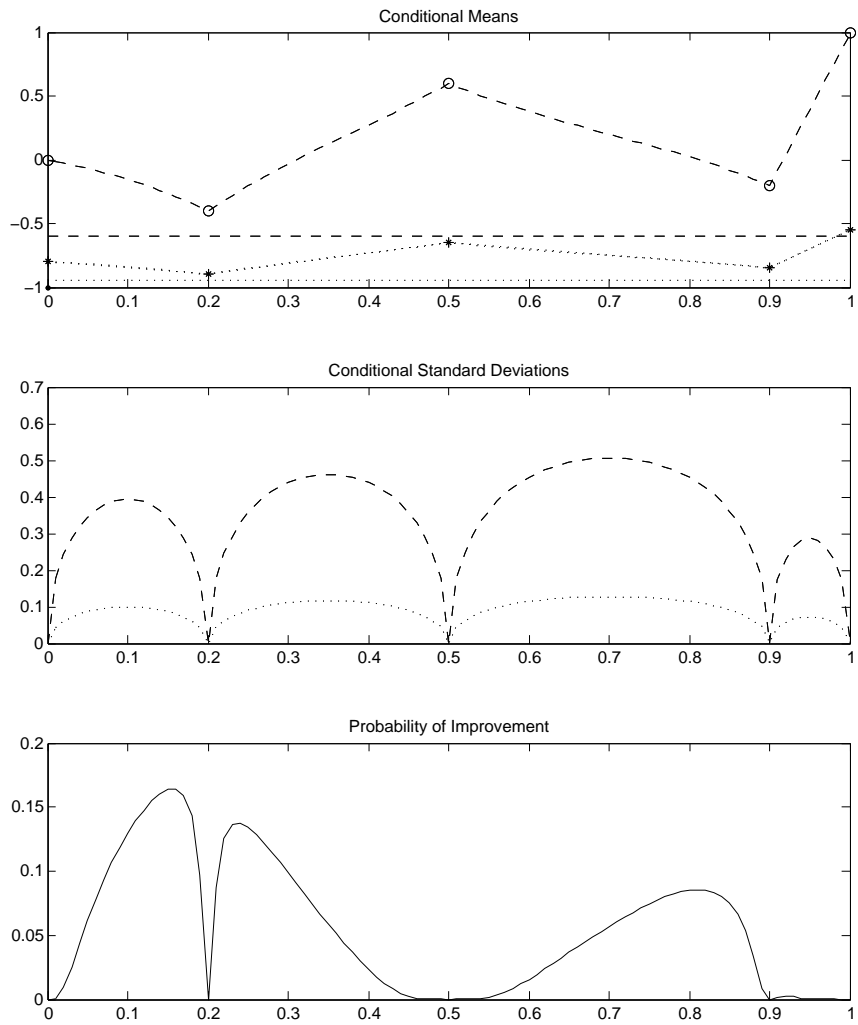}
\caption{An example of data used for planning the current iteration of the P-algorithm }
\label{fig:1}
\end{figure}

In spite of the obvious difference in the data, the functions expressing the probability of improvement for both cases coincide. Therefore, their maximizers which define the next points of function evaluations also coincide. This coincidence is implied by the strong homogeneity of the P-algorithm and the following relation: $ \phi(x)=af(x)+b$, where the values of $a,b$ up to five decimal digits are equal to $a=3.9765, b=3.1804$.

\section{Conclusions}
Both the P-algorithm and the one-step Bayesian algorithm are strongly homogeneous. The optimization results by these algorithms are invariant with respect to affine scaling of values of the objective function. The implementations of these algorithms using the conventional computer arithmetic combined with the scaling of function values, using the arithmetic of infinity, are applicable to the objective functions with either infinite  or infinitesimal values. The optimization results, obtained in this way, would be identical with the results obtained applying the implementations of the algorithms in the arithmetic of infinity.

\section{ACKNOWLEDGEMENTS}
The valuable remarks of two unknown referees facilitated a significant improvement of the presentation of results.



%% References


\bibliographystyle{model1a-num-names}



\end{document}